\documentclass[12pt,twoside]{amsart}
\usepackage{amssymb}
\usepackage{verbatim}
\usepackage{amsmath}
\usepackage{bm}
\usepackage{a4wide}
\usepackage[latin1]{inputenc}
\usepackage[T1]{fontenc}
\usepackage{times}
\usepackage{amssymb,latexsym}
\usepackage{enumerate}
\usepackage{color}
\usepackage[colorlinks,linkcolor=black, citecolor=black,urlcolor=black]{hyperref}

\makeatletter
\newcommand{\sumprime}{\if@display\sideset{}{'}\sum%
            \else\sum'\fi}
\makeatother

\begin{document}

\numberwithin{equation}{section}

\newtheorem{theorem}{Theorem}[section]
\newtheorem{proposition}[theorem]{Proposition}
\newtheorem{conjecture}[theorem]{Conjecture}
\def\theconjecture{\unskip}
\newtheorem{corollary}[theorem]{Corollary}
\newtheorem{lemma}[theorem]{Lemma}
\newtheorem{observation}[theorem]{Observation}
\newtheorem{definition}{Definition}
\numberwithin{definition}{section} 
\newtheorem{remark}{Remark}
\def\theremark{\unskip}
\newtheorem{kl}{Key Lemma}
\def\thekl{\unskip}
\newtheorem{question}{Question}
\def\thequestion{\unskip}
\newtheorem{example}{Example}
\def\theexample{\unskip}
\newtheorem{problem}{Problem}
\newtheorem*{problem*}{Problem}

\thanks{Supported by National Natural Science Foundation of China, No. 12271101}

\address [Yuanpu Xiong] {Department of Mathematical Sciences, Fudan University, Shanghai, 200433, China}
\email{ypxiong18@fudan.edu.cn}

\address [Zhiyuan Zheng] {Department of Mathematical and Computer Sciences, Tongling University, Anhui, 244000, China}
\email{2023052@tlu.edu.cn}

\title{Bergman functions on weakly uniformly perfect domains}
\author{Yuanpu Xiong and Zhiyuan Zheng}

\date{}

\begin{abstract}
We contruct two classes of Zalcman-type domains, on which the Bergman functions have certain pre-described boundary behaviors. Such examples also lead to generalizations of uniformly perfectness in the sense of Pommerenke. These weakly uniformly perfect conditions can be characterized in terms of the logarithm capacity. We obtain lower estimates for the boundary behaviors of Bergman kernel functions on such domains.
\end{abstract}

\maketitle

\tableofcontents

\section{Introduction}
A bounded domain $\Omega\subset\mathbb{C}^n$ is said to be Bergman exhaustive if the Bergman kernel function $K_\Omega(z)$ is exhaustive, while it is called Bergman complete if the Bergman metric is complete, i.e., the Bergman distance $d_\Omega$ is complete. The exhaustiveness and completeness are two central topics in the study of Bergman functions (i.e., the Bergman kernel, metric and distance). There is a large literature in these directions (see, e.g., \cite{Ohsawa84, JP, JPZ, Zwonek99, Chen99, BP98, Herbort, PflugZwonek05}). In particular, it is known that if $\Omega$ is hyperconvex, then it is Bergman exhaustive and Bergman complete.

One can also study Bergman exhaustiveness and completeness quantitatively. For example, there are many lower estimates for the Bergman kernel implying Bergman exhaustiveness in different settings. After some early works of Diederich \cite{Diederich70,Diederich73} and Catlin \cite{Catlin84,Catlin89}, Diederich-Ohsawa \cite{DO95} obtained an effective estimate concerning Bergman completeness on a bounded pseudoconvex domain $\Omega$ with $C^2$ boundary. They showed that
\begin{equation}\label{eq:distance_estimate_1}
d_\Omega(z,z_0)\gtrsim\log\log\frac{1}{\delta_\Omega(z)},\ \ \ z\rightarrow\partial\Omega,
\end{equation}
where $z_0\in\Omega$ is fixed, and $\delta_\Omega(z)$ is the Euclidean distance from $z\in\Omega$ to $\partial\Omega$. Their result is actually proved in a larger class of bounded pseudoconvex domains. B{\l}ocki \cite{Blocki05} improved the estimate \eqref{eq:distance_estimate_1} to
\begin{equation}\label{eq:distance_estimate_2}
d_\Omega(z,z_0)\gtrsim\frac{\log\frac{1}{\delta_\Omega(z)}}{\log\log\frac{1}{\delta_\Omega(z)}},\ \ \ z\rightarrow\partial\Omega
\end{equation}
under a slightly stronger condition (which is also satisfied by bounded pseudoconvex domains with $C^2$ boundaries). We also refer the reader to \cite{Ohsawa99, Herbort16, Chen17} for some applications and generalizations. The estimates \eqref{eq:distance_estimate_1} and \eqref{eq:distance_estimate_2} are not yet known to be sharp in general. But for a planar domain with $C^2$ boundary, we see from Diederich's works \cite{Diederich70,Diederich73} that the sharp boundary behavior is
\[
d_\Omega(z,z_0)\asymp\log\frac{1}{\delta_\Omega(z)}.
\]

The goal of this paper is to show that both the Diederich-Ohsawa type estimate and the B{\l}ocki type estimate could really exist for some planar domains. That is, to construct certain bounded domain whose Bergman distance has the pre-described boundary behavior in \eqref{eq:distance_estimate_1} or \eqref{eq:distance_estimate_2}. Let us consider the Zalcman-type domain (cf. \cite{Zalcman})
\begin{equation}\label{eq:Zalcman}
\Omega := D(0,1)\setminus\left(\bigcup_{k=1}^{\infty}\overline{D(x_k,r_k)}\cup \{0\}\right).
\end{equation}
Here, $x_k\in(0,1)$ and $0<r_k\ll{x_k}$ so that the discs $D(x_k,r_k)$ are pairwise disjoint. In this paper, we set
\begin{equation}\label{eq:Zalcman_h}
r_k=x_{k+1}=h(x_k)\ \ \ \text{and}\ \ \ 0<x_1\ll1,
\end{equation}
where $h:(0,\varepsilon_0)\rightarrow(0,\infty)$ is an increasing function with $h(r)=o(r)$ as $r\rightarrow0+$, and $\varepsilon_0>0$ is some constant. We mainly consider the following two types of $h$:
\begin{itemize}
\item[$(1)$] $h_{1,\alpha}(r)=r^\alpha$,
\item[$(2)$] $h_{2,\beta}(r)=r(\log(1/r))^{-\beta}$.
\end{itemize}
The following result implies that the boundary behaviors of the Bergman distances in the estimates \eqref{eq:distance_estimate_1} and \eqref{eq:distance_estimate_2} can be fulfilled with these choices of $h$.
\begin{theorem}\label{thm2}
Let $0<x\ll1$.
\begin{itemize}
\item[$(1)$]
If $h(x)=h_{1,\alpha}(x)$, then
\[
K_{\Omega}(-x) \asymp  \frac{1}{x^{2}\log \frac{1}{x}},\ \ \ d_{\Omega}(-x_1,-x)\asymp \log \log \frac{1}{x}.
\]

\item[$(2)$]
If $h(x)=h_{2,\beta}(x)$, then
\[
K_{\Omega}(-x) \asymp \frac{1}{x^2\log \log \frac{1}{x}},\ \ \ d_{\Omega}(-x_1,-x) \asymp \frac{\log \frac{1}{x}}{\log\log \frac{1}{x}}.
\]
\end{itemize}
\end{theorem}

\begin{remark}
{\rm As a simple consequence of Wiener's criterion (cf. \cite[Theorem 5.4.1]{Ransford}), the domains in Theorem \ref{thm2} are hyperconvex, so that they are Bergman complete.}
\end{remark}

The Zalcman-type domains are very useful to construct examples in the study of Bergman exhaustiveness and completeness (see, e.g., \cite{Chen99, Zwonek01, Jucha}). Indeed, Theorem \ref{thm2} is largely inspired by the work of Jucha \cite{Jucha}, where various techniques are applied to study the lower and upper bound for the Bergman functions. It is usually easier to obtain sharp estimates for the Bergman functions on one-dimensional domains, since it is easier to construct holomorphic functions and some integral representation (such as Cauchy's integral formula) can be applied. It would be a challenging question to consider the analogue of Theorem \ref{thm2} in high dimensions, i.e., to construct nontrivial examples in $\mathbb{C}^n$ such that the estimates \eqref{eq:distance_estimate_1} and \eqref{eq:distance_estimate_2} are sharp.

We also want to find some generalities in the examples in Theorem \ref{thm2}. It seems that they can be related to domains with uniformly perfect boundaries, a concept introduced by Pommerenke \cite{Pommerenke}. Let us consider the following generalization. In what follows, we always assume that $\Omega\subset\mathbb{C}$ and $\Omega\neq\mathbb{C}$ ($\Omega$ is not necessarily bounded).

\begin{definition}\label{def:h}
Let $h:(0,\varepsilon_0)\rightarrow(0,\infty)$ be an increasing function with $h(r)=o(r)$ as $r\rightarrow0+$, where $\varepsilon_0>0$. A domain $\Omega\subset\mathbb{C}$ is said to have $h-$uniformly perfect boundary or weakly uniformly perfect boundary if there exists some $c,r_0>0$ such that
\[
\{z\in\mathbb{C};\ c\cdot{h(r)} \leq |z-a| \leq r\}\cap\partial\Omega\neq\emptyset.
\]
\end{definition}

We say $\Omega$ satisfies the condition $(U)_{1,\alpha}$ or $(U)_{2,\beta}$ if it has $h_{1,\alpha}$ or $h_{2,\beta}-$uniformly perfect boundaries, respectively. It follows that

\begin{proposition}\label{prop:Zalcman_weak_uniformly_perfect}
\begin{itemize}
\item[$(1)$]
The Zalcman-type domain in Theorem \ref{thm2}/(1) satisfies the condition $(U)_{1,\alpha}$, but does not satisfy the condition $(U)_{1,\alpha-\varepsilon}$ for every $\varepsilon>0$.

\item[$(2)$]
The Zalcman-type domain in Theorem \ref{thm2}/(2) satisfies the condition $(U)_{2,\beta}$, but does not satisfy the condition $(U)_{2,\beta-\varepsilon}$ for every $\varepsilon>0$.
\end{itemize}
\end{proposition}

We will prove Proposition \ref{prop:Zalcman_weak_uniformly_perfect} in \S\,\ref{sec:weak_uniform_perfect}.

If $\Omega$ is bounded and we take $h(r)=r$ in Definition \ref{def:h} , then $\partial\Omega$ is uniformly perfect. Uniformly perfectness is deeply connected with many questions in complex analysis, dynamics and geometry (cf. \cite{Fernandez,Jarvi,Lithner,Osgood,Rocha,JianhuaZheng}, etc.). Pommerenke also showed that uniformly perfectness can be characterized by capacity conditions. Following his idea, let us consider the following condition on $\partial\Omega$:
\begin{itemize}
\item[$(C)_h$:] There exist constants $C,r_0>0$, such that
\[
\mathrm{Cap}(\overline{D}(a,r)\setminus\Omega)\geq{C\cdot{h(r)}}
\]
for all $a\in\partial\Omega$ and $r\in(0,r_0)$.
\end{itemize}
For simplicity, we denote $(C)_{h_{1,\alpha}}$ and $(C)_{h_{2,\beta}}$ by $(C)_{1,\alpha}$ and $(C)_{2,\beta}$, respectively. Pommerenke proved in \cite{Pommerenke} that a domain $\Omega$ has uniformly perfect boundary if and only if the condition $(C)_h$ holds with $h(r)=r$.

As for weakly uniformly perfectness, we have the following relationships between the conditions $(U)_{1,\alpha}$, $(C)_{1,\alpha}$, $(U)_{2,\beta}$ and $(C)_{2,\beta}$.

\begin{theorem}\label{th:weakly_uniform_perfect}
Let $\Omega$ be a domain in $\mathbb{C}$.

\begin{itemize}
\item[$(1)$] $(C)_{1,\alpha}\Rightarrow(U)_{1,\alpha}$. Conversely, if $1<\alpha<2$, then $(U)_{1,\alpha}\Rightarrow(C)_{1,(2-\alpha)^{-1}}$.

\item[$(2)$] $(U)_{2,\beta}\Leftrightarrow(C)_{2,\beta}$.
\end{itemize}
\end{theorem}

In contrast to Theorem \ref{th:weakly_uniform_perfect}, for every $\alpha\geq2$, there exists a Cantor-type set $\mathcal{C}$, with $\mathrm{Cap}(\mathcal{C})=0$ and $\Omega:=\mathbb{C}\setminus\mathcal{C}$ satisfying $(U)_{1,\alpha}$. Thus the condition $1<\alpha<2$ in Theorem \ref{th:weakly_uniform_perfect}/(1) cannot be removed. More details of the construction will be given in \S \ref{sec:weak_uniform_perfect}.

If a domain $\Omega$ satisfies $(U)_{1,\alpha}$ with $1<\alpha<2$, or $(U)_{2,\beta}$ with $\beta>0$, then $\mathbb{C}\setminus\Omega$ is non-polar. It follows from Carleson's theorem (cf. \cite{Carleson}) that $A^2(\Omega)\neq\{0\}$. Moreover, the Bergman kernel is strictly positive, and Bergman metric exists (see, e.g., \cite[Theorem 4]{BlockiZwonek}).

It would be an interesting question to study Bergman functions on these weakly uniformly perfect domains. In particular, it is natural to ask

\begin{problem}\label{prob}
If $\Omega$ satisfies the condition $(U)_{1,\alpha}$ for some $1<\alpha<2$ but does not satisfy the condition $(U)_{1,\alpha-\varepsilon}$ for any $\varepsilon>0$, then does there exists a sequence $\{z_k\}\subset\Omega$ with $z_k\rightarrow\partial\Omega$ as $k\rightarrow\infty$ and
\[
d_\Omega(z_k,z_0)\asymp\log\log\frac{1}{\delta_\Omega(z_k)}
\]
for some fixed $z_0$? One may raise a similar question for the B{\l}ocki estimate \eqref{eq:distance_estimate_2}.
\end{problem}

Theorem \ref{th:weakly_uniform_perfect} allows us to apply certain potential theoretical methods to study Bergman functions (cf. \cite{Zwonek, PflugZwonek, BlockiZwonek}). Inspired by the work of Pflug-Zwonek \cite{PflugZwonek}, we have the following lower estimate for the Bergman kernels on weakly uniformly perfect domains, which might be a first step to Problem \ref{prob}.

\begin{theorem}\label{th:Bergman_kernel}
Let $\Omega$ be a domain in $\mathbb{C}$, and $w\in\Omega$ sufficiently close to the boundary.

\begin{itemize}
\item[$(1)$]
If $\Omega$ satisfies the condition $(U)_{1,\alpha}$ for $1<\alpha<2$, then
\begin{equation}\label{eq:lower_1}
K_{\Omega}(w) \gtrsim \frac{1}{\delta_{\Omega}(w)^{2}\log \frac{1}{\delta_{\Omega}(w)}}.
\end{equation}

\item[$(2)$]
If $\Omega$ satisfies the condition $(U)_{2,\beta}$ for $\beta>0$, then
\begin{equation}\label{eq:lower_2}
K_{\Omega}(w) \gtrsim \frac{1}{\delta_\Omega(w)^2\log \log \frac{1}{\delta_\Omega(w)}}.
\end{equation}
\end{itemize}
\end{theorem}

Theorem \ref{thm2} implies that these estimates are sharp. Similarly, the condition $0<\alpha<2$ in Theorem \ref{th:Bergman_kernel}/(1) cannot be removed, in view of Carleson's theorem (cf. \cite{Carleson}). Theorem \ref{th:Bergman_kernel}/(1) is also a direct consequence of Theorem 3 in \cite{PflugZwonek} for a bounded domain $\Omega\subset\mathbb{C}$. On the other hand, we shall take a unified approach based on the idea in \cite{PflugZwonek} to prove both two assertions. More details will be given in \S \ref{sec:main}.

We conclude the introduction by the following remark. Chen \cite{Chen} obtained a charaterization of uniformly perfectness by using the boundary behavior of Bergman functions. He proved that a domain $\Omega\subset\mathbb{C}$ has uniformly perfect boundary if and only if $K_\Omega(z)\asymp\delta_\Omega(z)^{-2}$ and $b_\Omega(z)\asymp\delta_\Omega(z)^{-1}$. Here $b_\Omega(z)|dz|$ is the Bergman metric on a planar domain. It is not clear whether we can find characterizarions for weakly uniformly perfect domains in terms of Bergman functions.

\section{Preliminaries}
We present a sketched introduction to the theory of logarithm capacity. Let $\mu$ be a finite Borel measure on $\mathbb{C}$. We define its potential to be the function
\[
p_\mu(z):=\int_{\mathbb{C}}\log|z-w|d\mu(w),\ \ \ z\in\mathbb{C}.
\]
We have $p_\mu(z)\in[-\infty,\infty)$, and it is a subharmonic function. The energy of $\mu$ is defined to be
\[
I(\mu):=\int_{\mathbb{C}}\int_{\mathbb{C}}\log|z-w|d\mu(z)d\mu(w)=\int_{\mathbb{C}}p_\mu(z)d\mu(z).
\]
It is possible that $I(\mu)=-\infty$. Indeed, $E$ is defined to be a polar set if $I(\mu)=-\infty$ for any nonzero funite Borel measure $\mu$ which is supported in $E$. We say certain property holds nearly everywhere (n.e.) on a subset $S\subset\mathbb{C}$, if it holds everywhere on $S\setminus{E}$, where $E$ is a polar set.

Let $E\subset\mathbb{C}$, and $\mathcal{P}(E)$ the collection of all Borel probability measures on $E$. Then the logarithm capacity of $E$ is defined to be
\[
\mathrm{Cap}(E):=\sup_{\mu\in\mathcal{P}(E)}e^{I(\mu)}.
\]
If $E$ is compact and non-polar, then there is a unique equilibrium measure on $E$, i.e., a Borel probability measure $\mu_E$ with $I(\mu_E)=\sup_{\mu\in\mathcal{P}(E)}I(\mu)$. Thus
\[
\mathrm{Cap}(E):=e^{I(\mu_E)}.
\]
Moreover, $\mu_E$ is supported in the exterior boundary of $E$ (see e.g., \cite[Theorem 3.7.6]{Ransford}).

\begin{example}
The equilibrium measure of $\overline{D(a,r)}$ is the normalized arclength measure on $\partial{D(a,r)}$. Moreover, $\mathrm{Cap}(D(a,r))=r$.
\end{example}

Another approach to compute or estimate the logarithm capacity is by using the transfinite diameter. For a compact set $E\subset\mathbb{C}$, we define the $n-$th diameter of $E$ by
\[
\delta_n(E):=\sup\left\{\prod_{1\leq{j<k}\leq{n}}|z_j-z_k|^{\frac{2}{n(n-1)}};\ z_1,\cdots,z_n\in{E}\right\}.
\]
A theorem of Fekete and Szeg\"{o} (cf. \cite{Ransford}, Theorem 5.5.2) asserts that $\delta_n(E)$ is decreasing with respect to $n$, and
\begin{equation}\label{eq:Fekete-Szego}
\lim_{n\rightarrow\infty}\delta_n(E)=\mathrm{Cap}(E).
\end{equation}
The limit is also called the transfinite diameter of $E$.

Let $E$ be compact and $T:E\rightarrow\mathbb{C}$ a map with
\[
|T(z)-T(w)|\leq{A|z-w|^c},
\]
where $A>0$ and $0<c\leq1$. By using \eqref{eq:Fekete-Szego}, one can verify that
\begin{equation}\label{eq:cap_map}
\mathrm{Cap}(T(E))\leq{A}\mathrm{Cap}(E)^c.
\end{equation}
In particular, if $T$ is the dilatation map $z\mapsto{tz}$ with $t>0$, then
\begin{equation}\label{eq:cap_dilatation}
\mathrm{Cap}(T(E))=\mathrm{Cap}(tE)=t\mathrm{Cap}(E).
\end{equation}
Moreover, for any Borel set in $\mathbb{C}$,
\begin{equation}\label{eq:measure_dilatation}
\mu_{tE}(B)=\mu_E(t^{-1}B).
\end{equation}

More properties of logarithm capacity can be found in \cite{Ransford}, Chapter 5. In particular, we shall make use of the following inequality
\begin{equation}\label{eq:cap_sum}
\frac{1}{\log\left(d/\mathrm{Cap}(E)\right)}\leq\sum_n\frac{1}{\log\left(d/\mathrm{Cap}(E_n)\right)},
\end{equation}
where $\{E_n\}$ is a sequence of Borel subsets in $\mathbb{C}$, $E=\bigcup_n{E_n}$, and $d>0$ with $\mathrm{diam}\,(E)\leq{d}$ and $\mathrm{Cap}\,(E)\leq{d}$.

\section{Bergman Functions on Zalcman Type Domains}
Let us first prove the following technical lemma.

\begin{lemma}\label{lem1}
Let $0<r<R<\infty$. Then there exists a smooth function $\varphi$, with $\varphi\equiv1$ when $|z|\leq{r}$, $\varphi\equiv0$ when $|z|\geq{R}$, and
\[
\int_{\mathbb{C}}\left|\frac{\partial\varphi}{\partial\bar{z}}\right|^2\lesssim\left(\log\frac{R}{r}\right)^{-1}.
\]
\end{lemma}
\begin{proof}
Let $\chi:\mathbb{R}\rightarrow[0,1]$ be a smooth function with $\chi|_{(-\infty,0]}\equiv1$, and $\chi|_{[1,+\infty)}\equiv0$. Consider
\[
\varphi(z):=\chi\left(\frac{\log|z|-\log{r}}{\log{R}-\log{r}}\right).
\]
Then $\varphi|_{\{|z|\leq{r}\}}=1$, $\varphi|_{\{|z|\geq R\}}=0$, and
\[
\int_{\mathbb{C}}\left|\frac{\partial\varphi}{\partial\bar{z}}\right|^2\lesssim\left(\log\frac{R}{r}\right)^{-2}\int_{r<|z|<R}\frac{1}{|z|^2}\asymp\left(\log\frac{R}{r}\right)^{-1}.
\]
\end{proof}

Let $K_\Omega(z)$ be the Bergman kernel function and $b_\Omega(z)|dz|$ be the Bergman metric for a planar domain $\Omega$. Recall that
\[
K_\Omega(z)=\sup\left\{|f(z)|^2;\ f\in{A^2(\Omega)},\ \|f\|_{L^2(\Omega)}=1\right\}
\]
and
\[
b_\Omega(z)=K_\Omega(z)^{-1/2}\sup\left\{|f'(z)|;\ f\in{A^2(\Omega)},\ f(z)=0,\ \|f\|_{L^2(\Omega)}=1\right\}.
\]

\subsection{Proof of Theorem \ref{thm2}/(1)}\label{subsec:alpha}
Let $\Omega$ be the Zalcman-type domain \eqref{eq:Zalcman} defined by $h=h_{1,\alpha}$. Then we have
\begin{equation}\label{eq4.2}
x_{k+1}\ll{x_k},\ \ \ r_k\ll{x_k},
\end{equation}
and
\begin{equation}\label{eq4.3}
\log\frac{1}{x_k}\asymp\log\frac{1}{x_{k+1}}\asymp\log\frac{1}{r_k}\asymp\frac{1}{\alpha^k}.
\end{equation}
Here and in what follows, the implicit constants can depend only on $\alpha$ and $x_1$.

We divide the proof into the four parts.

(i) By Proposition \ref{prop:Zalcman_weak_uniformly_perfect}/(1), Theorem \ref{th:Bergman_kernel} can be applied to $\Omega$ when $1<\alpha<2$ to obtain the lower estimate for Bergman kernel. In general, given $0<x\ll1$, take an integer $k$ with $x\in(x_{k+1},x_k)$ and consider the function
\[
f(z):=\frac{1}{z-{x_{k+1}}}.
\]
It follows that
\begin{equation}\label{eq4.4}
\|f\|_{L^2(\Omega)}^2\leq\int_{r_{k+1}<|z-x_{k+1}|<2}\frac{1}{|z-x_{k+1}|^2}\lesssim\log\frac{1}{r_k},
\end{equation}
and hence
\begin{equation}\label{eq4.5}
K_\Omega(-x)\geq\frac{|f(-x)|^2}{\|f\|_{L^2(\Omega)}^2}\gtrsim\frac{1}{x^2\log\frac{1}{x}}.
\end{equation}
in view of \eqref{eq4.3}.

(ii) Next, we consider the upper estimate for $K_\Omega$. For later usage, we consider $K_\Omega(w)$, where $w\in\Omega$ and $x_{k+1}<|w|<x_k$, instead of the special case $w=-x\in(-x_k,-x_{k+1})$. Set
\[
\Omega_k:=D\left(0,\frac{4}{5}\right)\setminus\left(\bigcup^{k+1}_{l=1}\overline{D(x_l,2r_l)}\cup\overline{D(0,2x_{k+2})}\right).
\]
By \eqref{eq4.2}, the closed discs $\overline{D(x_l,2r_l)},\ 1 \le l \le k+1$ and $\overline{D(0,2x_{k+2})}$ are pairwise disjoint. Moreover, $\overline{\Omega}_k\subset\Omega$. Then we can apply Cauchy's integral formula to obtain
\begin{equation}\label{eq4.6}
f(w)=\frac{1}{2\pi{i}}\int_{\partial\Omega_k}\frac{f(z)}{z-w}dz,\ \ \ \forall\,f\in{A^2(\Omega)}.
\end{equation}

Following \cite{Jucha}, we take $\varphi_0\in{C^\infty_0(\mathbb{C})}$, such that $\varphi_0\equiv1$ when $|z|=4/5$ and
\[
\mathrm{supp}\,\varphi_0\subset\left\{\frac{3}{4}<|z|<\frac{5}{4}\right\}.
\]
For any $1\leq{l}\leq{k+1}$, set
\[
A_l:=\left\{2r_l<|z-x_l|<\frac{x_l}{4}\right\}\subset\Omega_k,
\]
so that they have pairwise disjoint closures. By Lemma \ref{lem1}, there exists a smooth function $\varphi_l$ with $\varphi_l|_{\{|z-x_l|\leq{2r_l}\}}\equiv1$, $\varphi_l|_{\{|z-x_l|\geq x_l/4\}}\equiv0$, and
\begin{equation}\label{eq4.7}
\left\|\frac{\partial\varphi_l}{\partial\bar{z}}\right\|_{L^2(\mathbb{C})}\lesssim\left(\log\frac{x_l}{r_l}\right)^{-\frac{1}{2}}.
\end{equation}
Moreover, the set
\[
\widetilde{A}_k:=\left\{2x_{k+2}<|z|<\frac{x_{k+1}}{4}\right\}\subset\Omega_k
\]
satisfies $\overline{\widetilde{A}}_k\cap{\overline{A}_l}=\emptyset$. By Lemma \ref{lem1} again, we have a smooth function $\widetilde{\varphi}_k$ with $\widetilde{\varphi}_k|_{\{|z|\leq{2x_{k+2}}\}}\equiv1$, $\widetilde{\varphi}_k|_{\{|z|\geq x_{k+1}/4\}}\equiv0$, and
\begin{equation}\label{eq4.8}
\left\|\frac{\partial\widetilde{\varphi}_k}{\partial\bar{z}}\right\|_{L^2(\mathbb{C})}\lesssim\left(\log\frac{x_{k+1}}{x_{k+2}}\right)^{-\frac{1}{2}}.
\end{equation}
We set
\[
\phi_k:=\varphi_0+\sum^{k+1}_{l=1}\varphi_l+\widetilde{\varphi}_k.
\]
It follows that $\phi_k\equiv1$ on $\partial\Omega_k$ and $\phi_k(w)=0$ when $x_k/3<|w|<2x_k/3$ and $k$ is sufficiently large. Thus we can apply Green's formula to \eqref{eq4.6}, i.e.,
\begin{equation}\label{eq4.9}
f(w)=\frac{1}{2\pi{i}}\int_{\partial\Omega_k}\frac{f(z)\phi_k(z)}{z-w}dz=-\frac{1}{2\pi{i}}\int_{\Omega_k}\frac{f(z)}{z-w}\frac{\partial\phi_k(z)}{\partial\bar{z}}dz\wedge{d\bar{z}}
\end{equation}
when $x_k/3<|w|<2x_k/3$. It follows that
\begin{eqnarray}
\notag |f(w)| &\lesssim& \int_{z\in\Omega_k}\frac{|f|}{|z-w|}\left|\frac{\partial\phi_k}{\partial\bar{z}}\right|\\
\notag &\leq&\int_{\frac{3}{4}<|z|<\frac{4}{5}}\frac{|f|}{|z-w|}\left|\frac{\partial\varphi_0}{\partial\bar{z}}\right|+\sum^{k+1}_{l=1}\int_{z\in{A_l}}\frac{|f|}{|z-w|}\left|\frac{\partial\varphi_l}{\partial\bar{z}}\right|+\int_{z\in\widetilde{A}_k}\frac{|f|}{|z-w|}\left|\frac{\partial\widetilde{\varphi}_k}{\partial\bar{z}}\right|\\
&=:&I_1+I_2+I_3, \ \ \ x_k/3<|w|<2x_k/3.\label{eq4.10}
\end{eqnarray}
Clearly, Cauchy-Schwarz inequality implies that
\begin{equation}\label{eq4.11}
I_1\lesssim \|f\|_{L^2(\Omega)}.
\end{equation}
As for $I_2$, since $x_{l+1}\ll{x_l}$ and $x_k/3<|w|<2x_k/3$, we infer from \eqref{eq4.7} and \eqref{eq4.3} that
\begin{eqnarray*}
I_2 &\lesssim& \frac{1}{|w|}\int_{z\in{A_{k+1}}}|f|\left|\frac{\partial\varphi_{k+1}}{\partial\bar{z}}\right|+\sum^k_{l=1}\frac{1}{x_l}\int_{z\in{A_l}}|f|\left|\frac{\partial\varphi_l}{\partial\bar{z}}\right|\\
&\leq& \frac{1}{|w|}\|f\|_{L^2(\Omega)}\left\|\frac{\partial\varphi_{k+1}}{\partial\bar{z}}\right\|_{L^2(\mathbb{C})}+\sum^k_{l=1}\frac{1}{x_l}\|f\|_{L^2(\Omega)}\left\|\frac{\partial\varphi_l}{\partial\bar{z}}\right\|_{L^2(\mathbb{C})}\\
&\lesssim& \frac{1}{|w|}\|f\|_{L^2(\Omega)} \frac{1}{(\log \frac{x_{k+1}}{r_{k+1}})^{\frac{1}{2}}}       +\sum^k_{l=1}\frac{1}{x_l}\|f\|_{L^2(\Omega)}     \frac{1}{(\log \frac{x_l}{r_l})^{\frac{1}{2}}}  \\
&\lesssim&\left(\frac{1}{|w|(\log\frac{1}{|w|})^{\frac{1}{2}}}+\sum^k_{l=1}\frac{1}{x_l(\log\frac{1}{x_l})^{\frac{1}{2}}}\right)\|f\|_{L^2(\Omega)}.
\end{eqnarray*}
By \eqref{eq4.2} and \eqref{eq4.3}, we have $x_{l+1}\ll{x_l}$ and $\log(1/x_{l+1})\asymp\log(1/x_l)$. Thus we may assume that
\[
\frac{1}{x_{l+1}(\log\frac{1}{x_{l+1}})^{\frac{1}{2}}}\geq\frac{2}{x_l(\log\frac{1}{x_l})^{\frac{1}{2}}}.
\]
Then
\begin{eqnarray}
\sum^k_{l=1}\frac{1}{x_l(\log\frac{1}{x_l})^{\frac{1}{2}}} &\leq& \frac{1}{x_k(\log\frac{1}{x_k})^{\frac{1}{2}}}\sum^k_{l=1}\frac{1}{2^{k-l}}\nonumber\\
&\asymp& \frac{1}{x_k(\log\frac{1}{x_k})^{\frac{1}{2}}}\nonumber\\
&\asymp& \frac{1}{|w|(\log\frac{1}{|w|})^{\frac{1}{2}}},\label{eq:sum_upper_alpha}
\end{eqnarray}
and hence
\begin{equation}\label{eq4.12}
I_2\lesssim\frac{\|f\|_{L^2(\Omega)}}{|w|(\log\frac{1}{|w|})^{\frac{1}{2}}}.
\end{equation}
For $I_3$, since $x_{k+1}\ll|w|$, we can proceed similarly to obtain
\begin{equation}\label{eq4.13}
I_3\lesssim\frac{1}{|w|}\int_{z\in\widetilde{A}_k}|f|\left|\frac{\partial\widetilde{\varphi}_k}{\partial\bar{z}}\right|\leq\frac{1}{|w|}\|f\|_{L^2(\Omega)}\left\|\frac{\partial\widetilde{\varphi}_k}{\partial\bar{z}}\right\|_{L^2(\mathbb{C})}\lesssim\frac{\|f\|_{L^2(\Omega)}}{|w|(\log\frac{1}{|w|})^{\frac{1}{2}}}.
\end{equation}
Notice that we used \eqref{eq4.8} and \eqref{eq4.3} in the last inequality. By \eqref{eq4.11}, \eqref{eq4.12} and \eqref{eq4.13}, we have
\[
|f(w)|\lesssim \frac{\|f\|_{L^2(\Omega)}}{|w|(\log\frac{1}{|w|})^{\frac{1}{2}}},\ \ \ x_k/3<|w|<2x_k/3,
\]
and hence
\begin{equation}\label{eq4.14}
K_\Omega(w)\lesssim\frac{1}{|w|^2\log\frac{1}{|w|}},\ \ \ x_k/3<|w|<2x_k/3.
\end{equation}
Moreover, the above argument also works for $-x\in(-x_k,-x_{k+1})$. Thus
\begin{equation}\label{eq4.14-1}
K_\Omega(-x)\lesssim\frac{1}{x^2\log\frac{1}{x}},\ \ \ x\in(0,x_1).
\end{equation}

(iii) It remains to find the boundary behavior of $d_\Omega$. Let us first consider the lower bound. Let $w\in\Omega$ with $x_{k+1}<|w|<x_k$. Following \cite{Jucha}, we consider the holomorphic function
\[
f(z)=\frac{1}{z-x_k}-\frac{w-x_{k+1}}{w-x_k}\cdot\frac{1}{z-x_{k+1}}.
\]
Clearly, $f(w)=0$. When $x_k/3<|w|<2x_k/3$ or $w=-x\in(-x_k,-x_{k+1})$, we have
\[
|f'(w)|=\left|\frac{x_k-x_{k+1}}{(w-x_{k+1})(w-x_k)^2}\right|\gtrsim\frac{1}{|w|x_k},
\]
and
\begin{eqnarray*}
\|f\|_{L^2(\Omega)}&\leq& \left\|\frac{1}{\cdot-x_k}\right\|_{L^2(\Omega)}+\left|\frac{w-x_{k+1}}{w-x_k}\right|\cdot\left\|\frac{1}{\cdot-x_{k+1}}\right\|_{L^2(\Omega)}\\
&\lesssim& \left(\log\frac{1}{r_{k+1}}\right)^{\frac{1}{2}}+\left(\log\frac{1}{r_k}\right)^{\frac{1}{2}}\\
&\asymp&\left(\log\frac{1}{|w|}\right)^{\frac{1}{2}},
\end{eqnarray*}
in view of \eqref{eq4.3}. This combined with \eqref{eq4.14} yields that
\[
b_\Omega(w)\geq\frac{|f'(w)|/\|f\|_{L^2(\Omega)}}{K_\Omega(w)^{\frac{1}{2}}}\gtrsim\frac{1}{x_k}
\]
when $x_k/3<|w|<2x_k/3$ or $w=-x\in(-x_k,-x_{k+1})$.

Let $x\in(x_{k+1},x_k)$. For any smooth curve $\gamma:[0,1]\rightarrow\Omega$ with $\gamma(0)=-x_1$ and $\gamma(1)=-x$, we can take some disjoint pieces $\gamma|_{[a_l,b_l]}$, ($ 1\leq{l}\leq{k-1}$), where $a_l<b_l<a_{l+1}$, $|\gamma(a_l)|=2x_l/3$, $|\gamma(b_l)|=x_l/3$, and
\[
\gamma([a_l,b_l])\subset\left\{\frac{x_l}{3}<|w|<\frac{2x_l}{3}\right\}.
\]
It follows that
\begin{eqnarray*}
\int_{\gamma|_{[a_l,b_l]}}b_\Omega(z)|dz| &=& \int^{b_l}_{a_l}b_\Omega(\gamma(t))|d\gamma(t)|\\
&\geq& \left|\int^{b_l}_{a_l}b_\Omega(\gamma(t))d|\gamma(t)|\right|\\
&\gtrsim& \int^{2x_k/3}_{x_k/3}\frac{1}{x_k}dr\\
&\geq& \frac{1}{3},
\end{eqnarray*}
and hence
\[
\int_\gamma b_\Omega(z)|dz|\gtrsim{k}.
\]
By definition, we have $x_k=x_1^{-\alpha^k}$, so that $k=\log\log\frac{1}{x}$. Since $\gamma$ is arbitrary,  we obtain the desired lower estimate for Bergman distance:
\begin{equation}\label{eq4.15}
d_\Omega(-x_1,-x)\gtrsim\log\log\frac{1}{x}.
\end{equation}

(iv) The upper estimate for $b_\Omega$ will also be obtained by using Cauchy's integral formula. Let $\Omega_k$, $\varphi_0$, $\varphi_l$, $\widetilde{\varphi}_k$ and $\varphi_k$ be as above, and $x\in(x_{k+1},x_k)$. It follows from Cauchy's integral formula and Green's formula that
\begin{eqnarray*}
f'(-x) &=& \frac{1}{2\pi{i}}\int_{\partial\Omega_k}\frac{f(z)}{(z+x)^2}dz=\frac{1}{2\pi{i}}\int_{\partial\Omega_k}\frac{f(z)\phi_k(z)}{(z+x)^2}dz\\
&=&-\frac{1}{2\pi{i}}\int_{\Omega_k}\frac{f(z)}{(z+x)^2}\frac{\partial\phi_k(z)}{\partial\bar{z}}dz\wedge{d\bar{z}}.
\end{eqnarray*}
If $f(-x)=0$, then we infer from \eqref{eq4.9} that
\[
\int_{\Omega_k}\frac{f(z)}{z+x}\frac{\partial\phi_k(z)}{\partial\bar{z}}dz\wedge{d\bar{z}}=0.
\]
Thus for any $f\in{A^2(\Omega)}$ with $f(-x)=0$, we have
\begin{eqnarray}
f'(-x) &=& -\frac{1}{2\pi{i}}\int_{\Omega_k}f(z)\left(\frac{1}{(z+x)^2}-\frac{1}{x(z+x)}\right)\frac{\partial\phi_k(z)}{\partial\bar{z}}dz\wedge{d\bar{z}}\nonumber\\
&=&\frac{1}{2\pi{i}}\int_{\Omega_k}\frac{f(z)}{(z+x)^2}\frac{z}{x}\frac{\partial\phi_k(z)}{\partial\bar{z}}dz\wedge{d\bar{z}}.\label{eq4.16}
\end{eqnarray}
As in the upper estimate for Bergman kernel, we have
\begin{eqnarray}
\notag |f(w)| &\lesssim& \int_{z\in\Omega_k}\frac{|f|}{|z+x|^2}\frac{|z|}{x}\left|\frac{\partial\phi_k}{\partial\bar{z}}\right|\\
\notag &\leq&\int_{\frac{3}{4}<|z|<\frac{4}{5}}\frac{|f|}{|z+x|^2}\frac{|z|}{x}\left|\frac{\partial\varphi_0}{\partial\bar{z}}\right|+\sum^{k+1}_{l=1}\int_{z\in{A_l}}\frac{|f|}{|z+x|^2}\frac{|z|}{x}\left|\frac{\partial\varphi_l}{\partial\bar{z}}\right|\\
\notag& &+\int_{z\in\widetilde{A}_k}\frac{|f|}{|z+x|^2}\frac{|z|}{x}\left|\frac{\partial\widetilde{\varphi}_k}{\partial\bar{z}}\right|\\
&=:&I_4+I_5+I_6. \label{eq4.17}
\end{eqnarray}
Clearly,
\begin{equation}\label{eq4.18}
I_4\lesssim\frac{\|f\|_{L^2(\Omega)}}{x}.
\end{equation}
For $I_5$, analogously to the estimate for $I_2$, we have
\begin{eqnarray}
I_5 &\lesssim& \frac{x_{k+1}}{x^3}\int_{z\in{A_{k+1}}}|f|\left|\frac{\partial\varphi_{k+1}}{\partial\bar{z}}\right|+\sum^k_{l=1}\frac{1}{x_lx}\int_{z\in{A_l}}|f|\left|\frac{\partial\varphi_l}{\partial\bar{z}}\right|\nonumber\\
&\leq& \frac{x_{k+1}}{x^3}\|f\|_{L^2(\Omega)}\left\|\frac{\partial\varphi_{k+1}}{\partial\bar{z}}\right\|_{L^2(\mathbb{C})}+\sum^k_{l=1}\frac{1}{x_lx}\|f\|_{L^2(\Omega)}\left\|\frac{\partial\varphi_l}{\partial\bar{z}}\right\|_{L^2(\mathbb{C})}\nonumber\\
&\lesssim& \left(\frac{x_{k+1}}{x^3(\log\frac{1}{x})^{\frac{1}{2}}}+\frac{1}{x}\sum^k_{l=1}\frac{1}{x_l(\log\frac{1}{x_l})^{\frac{1}{2}}}\right)\|f\|_{L^2(\Omega)}\nonumber\\
&\lesssim& \left(\frac{x_{k+1}}{x^2}+\frac{1}{x_k}\right)\frac{\|f\|_{L^2(\Omega)}}{x(\log\frac{1}{x})^{\frac{1}{2}}}\label{eq4.19}
\end{eqnarray}
in view of \eqref{eq4.3}, \eqref{eq4.7} and \eqref{eq:sum_upper_alpha}. Similarly, $I_6$ satisfies
\begin{equation}\label{eq4.20}
I_6\lesssim\frac{x_{k+1}}{x^3}\|f\|_{L^2(\Omega)}\left\|\frac{\partial\widetilde{\varphi}_k}{\partial\bar{z}}\right\|_{L^2(\mathbb{C})}\lesssim\frac{x_{k+1}\|f\|_{L^2(\Omega)}}{x^3(\log\frac{1}{x})^{\frac{1}{2}}}.
\end{equation}
Combine \eqref{eq4.18}, \eqref{eq4.19} with \eqref{eq4.20}, we obtain
\[
\frac{|f'(-x)|}{\|f\|_{L^2(\Omega)}}\lesssim\left(\frac{x_{k+1}}{x^2}+\frac{1}{x_k}\right)\frac{1}{x(\log\frac{1}{x})^{\frac{1}{2}}},\ \ \ x\in(x_{k+1},x_k),
\]
for all $f\in{A^2(\Omega)}$ with $f(-x)=0$. By using \eqref{eq4.5}, we have
\[
b_\Omega(-x)\lesssim\frac{1}{K_\Omega(-x)^{\frac{1}{2}}}\left(\frac{x_{k+1}}{x^2}+\frac{1}{x_k}\right)\frac{1}{x(\log\frac{1}{x})^{\frac{1}{2}}}\lesssim\frac{x_{k+1}}{x^2}+\frac{1}{x_k},
\]
which implies the desired upper estimate
\begin{eqnarray} \label{eq4.21}
d_\Omega(-x_1,-x) &\leq& \sum^k_{l=1}\int^{x_l}_{x_{l+1}}b_\Omega(t)dt\lesssim\sum^k_{l=1}\left(2-\frac{2x_{l+1}}{x_l}\right)\lesssim{k}\nonumber\\
&\asymp& \log\log\frac{1}{x}.
\end{eqnarray}
Now we complete the proof of Theorem \ref{thm2}, with the four estimate obtained in \eqref{eq4.5}, \eqref{eq4.14-1}, \eqref{eq4.15} and \eqref{eq4.21}.

\subsection{Proof of Theorem \ref{thm2}/(2)}
Let $\Omega$ be the Zalcman-type domain \eqref{eq:Zalcman} defined by $h=h_{2,\beta}$. Thus the sequence $\{x_k\}$ satisfies
\begin{equation}\label{eq5.1}
\log \frac{1}{x_{k+1}} = \log \frac{1}{x_k}+\beta \log \log \frac{1}{x_k} \asymp \log \frac{1}{x_{k}},
\end{equation}
so that
\begin{eqnarray}\label{eq:loglog_compare}
\log \log \frac{1}{x_{k+1}}
&=& \log \log \frac{1}{x_k}+\log\left(1+\frac{\log\log\frac{1}{x_k}}{\log\frac{1}{x_k}}\right)\nonumber\\
&=& \log \log \frac{1}{x_k}+O(1)\cdot\frac{\log\log\frac{1}{x_k}}{\log\frac{1}{x_k}}\nonumber\\
&=& \log \log \frac{1}{x_k}+o(1).
\end{eqnarray}
Moreover,
\begin{eqnarray*}
\frac{\log\frac{1}{x_{k+1}}}{\log\log\frac{1}{x_{k+1}}}
&=& \frac{\log \frac{1}{x_k}+\beta \log \log \frac{1}{x_k}}{\log\log\frac{1}{x_k}}\cdot\frac{\log\log\frac{1}{x_k}}{\log \log \frac{1}{x_k}+O(1)\cdot\frac{\log\log\frac{1}{x_k}}{\log\frac{1}{x_k}}}\\
&=& \left(\frac{\log\frac{1}{x_k}}{\log\log\frac{1}{x_k}}+\beta\right)\left(1-O(1)\cdot\frac{1}{\log\frac{1}{x_k}}\right)\\
&=& \frac{\log\frac{1}{x_k}}{\log\log\frac{1}{x_k}}+\beta+o(1),
\end{eqnarray*}
and hence
\begin{equation} \label{eq5.2}
\frac{\log \frac{1}{x_k}}{\log \log \frac{1}{x_k}}\asymp k.
\end{equation}

Here and in what follows, the implicit constants can depend only on $\beta$ and $x_1$. Since $r_k=x_{k+1}$, similar relations hold for $\{r_k\}$.

The proof will be also written in four parts.

(i) The lower estimate
\begin{equation}\label{eq5.9}
K_{\Omega}(-x) \gtrsim \frac{1}{x^2\log \log \frac{1}{x}}
\end{equation}
has been obtained in Theorem \ref{th:Bergman_kernel}/(2). We postpone the proof to \S\,\ref{sec:main}.

(ii) Next, we consider the upper estimate for $K_\Omega(w)$, where $x_k/3 <|w|<2x_k/3$ and $w=-x \in (-x_k,-x_{k+1})$. Let the domains $\Omega_k$, $A_l$, $\widetilde{A}_k$ and cut-off functions $\varphi_0$, $\varphi_l, \widetilde{\varphi}_k$ and $\phi_k$ be as in \ref{subsec:alpha}/(ii), but with $x_k,r_k$ defined by the function $h_{2,\beta}$. For $f\in{A^2(\Omega)}$, a similar application of Cauchy's integral formula and Green's formula gives
\begin{eqnarray*}
|f(w)| &\lesssim& \int_{\Omega_k} \frac{|f(z)|}{|z-w|}\left| \frac{\partial \varphi_0}{\partial \bar{z}} \right|  +\sum_{l=1}^{k+1}\int_{\Omega_k} \frac{|f(z)|}{|z-w|}\left| \frac{\partial \varphi_l}{\partial \bar{z}} \right| +\int_{\Omega_k} \frac{|f(z)|}{|z-w|}\left| \frac{\partial \widetilde{\varphi}_k}{\partial \bar{z}} \right|\\
&=:& I_1+I_2+I_3.
\end{eqnarray*}
$I_1$ still satisfies
\[
I_1 \lesssim \| f  \|_{L^2(\Omega)}.
\]
We also have $x_{k+1}\ll{x_k}$ in this case, so that \eqref{eq:loglog_compare} implies that
\begin{equation}\label{eq:double_beta}
\frac{1}{x_{l+1}(\log \log \frac{1}{x_{l+1}})^{\frac{1}{2}}} \geq \frac{2}{x_l(\log \log \frac{1}{x_l})^{\frac{1}{2}}},
\end{equation}
and hence
\begin{eqnarray*}
I_2 &\lesssim& \frac{1}{|w|}\|f\|_{L^2(\Omega)} \left\| \frac{\partial \varphi_{k+1}}{\partial \bar{z}} \right\|_{L^2(\mathbb{C})}+\sum_{l=1}^k \frac{1}{x_l}\|f\|_{L^2(\Omega)} \left\| \frac{\partial \varphi_{l}}{\partial \bar{z}} \right\|_{L^2(\mathbb{C})}\\
&\lesssim&\|f\|_{L^2(\Omega)} \left(  \frac{1}{|w|(\log \frac{x_{k+1}}{r_{k+1}})^{\frac{1}{2}}} +\sum_{l=1}^k\frac{1}{x_l(\log \frac{x_l}{r_l})^{\frac{1}{2}}} \right)  \\ 
&\lesssim&\|f\|_{L^2(\Omega)} \left(\frac{1}{|w|(\log \log \frac{1}{x_{k+1}})^{\frac{1}{2}}}+\frac{1}{x_k(\log \log\frac{1}{x_k})^{\frac{1}{2}}}\right)\\
&\lesssim& \frac{1}{|w|(\log \log \frac{1}{|w|})^{\frac{1}{2}}}\|f\|_{L^2(\Omega)}.
\end{eqnarray*}
Notice that we used \eqref{eq4.7} in the second inequality. Moreover,
\[
I_3 \lesssim \frac{1}{|w|(\log \frac{x_{k+1}}{x_{k+2}})^{\frac{1}{2}}}\|f\|_{L^2(\Omega)} \lesssim \frac{1}{|w|(\log \log \frac{1}{ |w|})^{\frac{1}{2}}}\|f\|_{L^2(\Omega)}.
\]
Thus we conclude that
\begin{equation}\label{eq5.3}
K_{\Omega}(w) \lesssim \frac{1}{|w|^2\log \log \frac{1}{|w|}},
\end{equation}
where $x_k/3<|w|<2x_k/3$ or $w=-x, x\in (x_{k+1},x_k)$.

(iii) We consider the lower estimate for $b_\Omega$. For $w\in \Omega$ with $x_k/3<|w|<2x_k/3$, set
\begin{equation}\label{eq5.15}
f(z)=\frac{1}{z-x_k}-\frac{a_k}{z-x_{k+1}}-\frac{(1-a_k)}{z-x_{k-1}} ,
\end{equation}
where $a_k \in \mathbb{C}$ such that $f(w)=0$. We have
\[
a_k=\frac{(x_{k-1}-x_k)(w-x_{k+1})}{(x_{k-1}-x_{k+1})(w-x_k)} \asymp \frac{x_{k-1}x_k}{x_{k-1}x_k} =O(1).
\]
Lengthy but straightforward computation yields
\begin{equation}\label{eq5.16}
f(z)=-\frac{(x_{k-1}-x_k)(x_k-x_{k+1})(w-z)}{(w-x_k)(z-x_k)(z-x_{k+1})(z-x_{k-1})}.
\end{equation}
and
\begin{eqnarray}
|f'(w)|&=&\left| -\frac{1}{(w-x_k)^2}+\frac{a_k}{(w-x_{k+1})^2}+\frac{1-a_k}{(w-x_{k-1})^2}\right|\nonumber\\
&=& \left| -\frac{1}{(w-x_k)^2}+\frac{x_{k-1}-x_k}{(x_{k-1}-x_{k+1})(w-x_k)(w-x_{k+1})}+\frac{1-a_k}{(w-x_{k-1})^2}\right| \nonumber\\
&=&\left| \frac{(x_k-x_{k+1})(x_k+x_{k+1}-w-x_{k-1})}{(x_{k-1}-x_{k+1})(w-x_k)^2(w-x_{k+1})}+\frac{1-a_k}{(w-x_{k-1})^2}\right|\nonumber \\
&\gtrsim&\left| \frac{(x_k-x_{k+1})(x_k+x_{k+1}-w-x_{k-1})}{(x_{k-1}-x_{k+1})(w-x_k)^2(w-x_{k+1})}\right|-\left|\frac{1}{(w-x_{k-1})^2}\right|\nonumber\\
&\gtrsim& \frac{1}{x_k^2}-\frac{1}{x_{k-1}^2}\nonumber\\
&\gtrsim& \frac{1}{x_k^2}, \label{eq5.17}
\end{eqnarray}
since $x_k\ll{x_{k-1}}$. We also need to find an upper bound for $\|f\|_{L^2(\Omega)}$. Divide the domain $\Omega$ into the following four parts:
\begin{eqnarray*}
\Omega_1 &=& \left\{ z ;\ r_{k+1}<|z-x_{k+1}|<\frac{x_{k+1}}{2}\right\},\\
\Omega_2 &=& \left\{ z ;\ r_{k}<|z-x_{k}|<\frac{x_{k}}{2}\right\},\\
\Omega_3 &=& \left\{ z ;\ r_{k-1}<|z-x_{k-1}|<\frac{x_{k-1}}{2}\right\},\\
\Omega_4 &=& \Omega \setminus (\Omega_1\cup \Omega_2 \cup \Omega_3).
\end{eqnarray*}
We also set
\begin{eqnarray*}
\Omega_4' &=& \left\{ z;\ \frac{x_{k-1}}{2} \le |z-x_{k-1}|\le 2x_{k-1}  \right\} \cap \left\{z;\ |z-x_{k+1}|\ge \frac{x_{k+1}}{2}\right\}\\
\Omega_4'' &=& \left\{ z;\ \frac{x_{k-1}}{2} \le |z-x_{k-1}|\le 2x_{k-1}  \right\} \cap \left\{z;\ |z-x_{k}|\ge \frac{x_{k}}{2}\right\}\\
\Omega_{4,m} &=& \{ z;\ mx_{k-1}\le |z-x_{k-1}| \le (m+1)x_{k-1}  \},\ \ \ m=2,3,4,\cdots,
\end{eqnarray*}
so that
\[
\Omega_4 \subset  \bigcup_{m=2}^{\infty}\Omega_{4,m}\cup(\Omega_4' \cap \Omega_4'').
\]
If $z\in \Omega_j$ ($j=1,2,3$), we infer from \eqref{eq5.16} that
\[
|f(z)| \asymp \frac{1}{|z-x_{k+2-j}|}.
\]
Thus
\begin{equation} \label{eq5.18}
\int_{\Omega_j}|f|^2 \asymp \log \frac{x_{k+2-j}}{r_{k+2-j}} \asymp \log \log \frac{1}{x_{k}}
\end{equation}
in view of \eqref{eq:loglog_compare}. On $\Omega_4' \cap \Omega_4''$, since $a_k=O(1)$, we have
\begin{eqnarray}
\|f\|_{L^2(\Omega_4' \cap \Omega_4'')}
 &\lesssim& \left\|\frac{1}{\cdot - x_k}\right\|_{L^2(\Omega_4'')}+\left\|\frac{1}{\cdot - x_{k+1}}\right\|_{L^2(\Omega_4')}+\left\|\frac{1}{\cdot - x_{k-1}}\right\|_{L^2\left(\left\{x_{k-1}/2\le |z-x_{k-1}| \le 2x_{k-1}\right\}\right)}\nonumber\\
 &\lesssim& \left\|\frac{1}{\cdot - x_k}\right\|_{L^2\left(\left\{x_{k}/2\le |z-x_{k}| \le 3x_{k-1}\right\}\right)}+\left\|\frac{1}{\cdot - x_{k+1}}\right\|_{L^2\left(\left\{x_{k-1}/2\le |z-x_{k+1}| \le 3x_{k-1}\right\}\right)}\nonumber\\
 & & +\left\|\frac{1}{\cdot - x_{k-1}}\right\|_{L^2\left(\left\{x_{k-1}/2\le |z-x_{k-1}| \le 2x_{k-1}\right\}\right)}\nonumber\\
 &\lesssim& \left(\log \frac{x_{k-1}}{x_k}\right)^{\frac{1}{2}}+\left(\log \frac{x_{k-1}}{x_{k+1}}\right)^{\frac{1}{2}}+(\log 4)^{\frac{1}{2}}\nonumber\\
 &\lesssim& \left(\log \log \frac{1}{x_k}\right)^{\frac{1}{2}}. \label{eq5.19}
\end{eqnarray}
Moreover, if $z\in\Omega_{4,m}$, we infer from \eqref{eq5.16} that
\[
|f(z)|\lesssim \frac{1}{m^2x_{k-1}},
\]
so that
\begin{equation}\label{eq5.20}
\int_{\Omega_{4,m}} |f|^2 \lesssim \frac{|\Omega_{4,m}|}{m^4x_{k-1}^2} \asymp \frac{1}{m^3}.
\end{equation}
By \eqref{eq5.18} \eqref{eq5.19} and \eqref{eq5.20}, we conclude that
\[
\int_{\Omega}|f|^2 \lesssim \log \log \frac{1}{x_k}.
\]
This together with \eqref{eq5.3} and \eqref{eq5.17} imply that
\begin{equation}\label{eq:b_lower_beta}
b_{\Omega}(w) \geq\frac{|f'(w)|/\|f\|_{L^2(\Omega)}}{K_\Omega(w)^{\frac{1}{2}}} \gtrsim \frac{1}{x_k}, \ \ \ \forall \, x_k/3<|w|<2x_k/3.
\end{equation}
By using the same method in \ref{subsec:alpha}/(iii), we conclude from \eqref{eq:b_lower_beta} and \eqref{eq5.2} that
\begin{equation}\label{eq5.21}
d_{\Omega}(-x_1,-x) \gtrsim k \asymp \frac{\log \frac{1}{x}}{\log \log \frac{1}{x}}.
\end{equation}

(iv) It remains to find an upper estimate for $b_\Omega$, which is analogous to the proof of \eqref{eq4.21}. Take $x\in (x_{k+1},x_k)$ and $f\in A^2(\Omega)$ with $f(-x)=0$. As in \eqref{eq4.17}, \eqref{eq4.18}, \eqref{eq4.19} and \eqref{eq4.20}, we have,
\begin{equation}\label{eq5.10}
|f'(-x)| \lesssim I_4+I_5+I_6,
\end{equation}
where
\begin{eqnarray}
I_4 &:=& \int_{\frac{3}{4}<|z|<\frac{4}{5}} \frac{|f(z)|}{|z+x|^2}\frac{|z|}{x} \left|  \frac{\partial \varphi_0}{\partial \bar{z}} \right| \lesssim \frac{1}{x}\|f\|_{L^2(\Omega)},\label{eq5.11}\\
I_5 &:=& \sum_{l=1}^{k+1}\int_{A_l} \frac{|f(z)|}{|z+x|^2}\frac{|z|}{x}\left|  \frac{\partial \varphi_l}{\partial \bar{z}} \right| \nonumber \\
 &\lesssim& \frac{x_{k+1}}{x^3} \|f\|_{L^2(\Omega)} \left\| \frac{\partial \varphi_{k+1}}{\partial \bar{z}} \right\|_{L^2(\mathbb{C})}+\sum_{l=1}^k \frac{1}{x_lx}\|f\|_{L^2(\Omega)} \left\|\frac{\partial \varphi_l}{\partial \bar{z}}\right\|_{L^2(\mathbb{C})}\nonumber\\
 &\lesssim& \left(\frac{x_{k+1}}{x^3(\log \log \frac{1}{x_k})^{\frac{1}{2}}} +\frac{1}{xx_k(\log\log \frac{1}{x_k})^{\frac{1}{2}}} \right)\|f\|_{L^2(\Omega)}\nonumber\\
&\lesssim& \left(\frac{x_{k+1}}{x^2}+\frac{1}{x_k}\right)\cdot \frac{\|f\|_{L^2(\Omega)}}{x(\log \log \frac{1}{x_k})^{\frac{1}{2}}}, \label{eq5.12}
\end{eqnarray}
and
\begin{eqnarray}
\notag I_6&:=&\int_{\tilde{A}_k}\frac{|f(z)|}{|z+x|^2}\frac{|z|}{x}\left|\frac{\partial \widetilde{\varphi}_k}{\partial \bar{z}}\right|\\
\notag &\lesssim& \frac{x_{k+1}}{x^3} \|f\|_{L^2(\Omega)} \left\| \frac{\partial \widetilde{\varphi}_k}{\partial \bar{z}} \right\|_{L^2(\Omega)}\\
&\lesssim& \frac{x_{k+1}}{x^2} \cdot \frac{\|f\|_{L^2(\Omega)}}{x(\log \log \frac{1}{x_k})^{\frac{1}{2}}},\label{eq5.13}
\end{eqnarray}
in view of \eqref{eq4.7}, \eqref{eq4.8}, \eqref{eq5.1}, \eqref{eq:loglog_compare}, \eqref{eq5.2} and \eqref{eq:double_beta}. By \eqref{eq5.10}, \eqref{eq5.11}, \eqref{eq5.12} and \eqref{eq5.13}, we have
\[
\frac{|f'(-x)|}{\|f\|_{L^2(\Omega)}} \lesssim \left( \frac{x_{k+1}}{x^2}+\frac{1}{x_k} \right) \cdot \frac{1}{x(\log \log \frac{1}{x_k})^{\frac{1}{2}}}
\]
for all $f\in{A^2(\Omega)}$ with $f(-x)=0$. By using \eqref{eq5.9}, it follows that
\[
b_{\Omega}(-x) \lesssim \frac{x_{k+1}}{x^2}+\frac{1}{x_k}, \ \ \  \forall \, x\in (x_{k+1},x_k),
\]
and hence
\[
d_{\Omega}(-x_{k+1},-x_k) \lesssim \int_{x_{k+1}}^{x_k} \left(\frac{x_{k+1}}{x^2}+\frac{1}{x_k}\right) =2-\frac{2x_{k+1}}{x_k}\le 2.
\]
Finally, we infer from \eqref{eq5.2} that
\begin{equation} \label{eq5.14}
d_{\Omega}(-x_1,-x) \lesssim k \asymp \frac{\log \frac{1}{x}}{\log \log \frac{1}{x}}, \ \ \  \forall \, x \in (x_{k+1},x_k).
\end{equation}

\section{Weakly Uniform Perfect Domains}\label{sec:weak_uniform_perfect}

\begin{proof}[Proof of Theorem \ref{th:weakly_uniform_perfect}]
We follow the idea in Pommerenke \cite{Pommerenke}. To simplify notations, let us denote
\[
h_1(r)=h_{1,\alpha}(r)= r^{\alpha}\ \ \ \text{and}\ \ \ h_2(r)=h_{2,\beta}(r):= r\left(\log\frac{1}{r}\right)^{-\beta}.
\]
Suppose that conditions $(U)_{1,\alpha}$ or $(U)_{2,\beta}$ fail. Then for any $c>0$ and $r_0>0$, there exists $a\in \partial \Omega$ and $r\in (0,r_0)$ such that
\[
\partial \Omega \cap \{ z\in \mathbb{C};\ c\cdot{h}_i(r)  \leq |z-a| \leq r \}\neq \emptyset,
\]
where $i=1$ or $2$. Thus for $r<\min \{r_1,\mathrm{diam}(\Omega)\}$, we must have
\[
\overline{D(a,r)}\setminus\Omega\subset D(a,c\cdot{h}_i(r)).
\] 
Recall that the capacity of a disc equals to its radius. By the compactness of $\overline{D(a,r)}\setminus\Omega$, we have
\[
\mathrm{Cap}(\overline{D(a,r)}\setminus\Omega)<c\cdot{h}_i(r),
\]
However, the constant $c$ can be arbitrarily small, so that $(C)_{1,\alpha}$ and $(C)_{2,\beta}$ would not hold. This proves $(C)_{1,\alpha}\Rightarrow(U)_{1,\alpha}$ and $(C)_{2,\beta}\Rightarrow(U)_{2,\beta}$.

For the other side, we fix some $a\in\Omega$, and take a sequence $\{s_k\}^{\infty}_{k=1}\subset(0,\infty)$ with $0<s_1\ll1$, so that
\begin{equation}\label{eq2.1}
s_{k+1}:=\frac{c}{5}h_i(s_k) \le \frac{1}{2}s_k.
\end{equation}
By the conditions $(U)_{1,\alpha}$ and $(U)_{2,\beta}$, for any positive integer $k$ and $z\in \partial \Omega$, there exists $\varphi_k(z)\in  \partial \Omega$, such that
\begin{equation}\label{eq2.2}
5s_{k+1}\le |\varphi_k(z)-z| \le s_k.
\end{equation}
For $k=1,2,\cdots$, let
\[
J_k:=\left\{(j_1, \cdots, j_k);\ j_l\in\{0,1\}, l=1,2,\cdots,k \right\}
\]
and $J_0:=\emptyset$.
We define a map
\[
\omega: \bigcup_{k=0}^{\infty} J_k \to \partial \Omega,
\]
inductively, with $\omega(\emptyset)=a \in \partial \Omega$ and
\begin{equation}\label{eq:omega_phi_k}
\omega(j_1,\cdots, j_{k+1})=\begin{cases}
\omega(j_1,\cdots, j_k), &\text{if}\ j_{k+1}=0, \\
\varphi_k(\omega(j_1,\cdots, j_k)), &\text{if}\ j_{k+1}=1.\\
\end{cases}
\end{equation}
Set $E_k:=\omega(J_k)\subset \Omega$. For $z=\omega(j_1, \cdots, j_k), z'=\omega(j_1', \cdots, j_k') \in E_k$ with $(j_1, \cdots, j_k)\ne (j_1',\cdots, j_k')$, we denote $m:=m(z,z')$ the maximal integer with $j_l=j_l', 1\le l \le m$, and consider $z^*=\omega(j_1, \cdots, j_m)$. We may assume that $j_{m+1}=0$ and $j_{m+1}'=1$. By \eqref{eq2.1}, \eqref{eq2.2} and \eqref{eq:omega_phi_k}, we see that
\begin{eqnarray*}
|z-z^*| &=& |\omega(j_1, \cdots, j_k)-\omega(j_1, \cdots, j_{m+1})| \\
&=&  \left|\sum_{t=m+1}^{k-1}\omega(j_1, \cdots, j_{t+1})-\omega(j_1, \cdots, j_{t})\right|\\
&\le& s_{m+1}+s_{m+2}+\cdots+s_{k-1}\\
&\le& s_{m+1}\left(1+\frac{1}{2}+\cdots+\frac{1}{2^{k-m-2}}\right)\\
&\le& 2s_{m+1}.
\end{eqnarray*}
Similarly, $|z'-\varphi_m(z^*)| \le 2s_{m+1}$. Hence
\begin{eqnarray}
|z-z'| &\ge& |z^*-\varphi_m(z^*)|-|z-z^*|-|z'-\varphi_m(z^*)|\nonumber\\
&\ge& 5s_{m+1}-4s_{m+1}\nonumber\\
&=&s_{m+1}\label{eq:z_z'}\\
&>&0.\nonumber
\end{eqnarray}
In particular, $E_k$ has exactly $2^k$ elements. Moreover, for any fixed $z\in E_k$, there are $2^{k-l-1}$ points $z'\in E_k$ with $m(z,z')=l$. It follows from \eqref{eq:z_z'} that
\begin{equation}\label{eq:transfinite_1}
\prod_{z'\in E_k,z'\ne z} |z-z'|\ge \prod_{l=0}^{k-1}(s_{l+1})^{2^{k-l-1}}.
\end{equation}
Since
\begin{eqnarray*}
|z-a|&\le& s_1+s_2+\cdots+s_k \\
&\le& s_1(1+\frac{1}{2}+\frac{1}{2^2} +\cdots+\frac{1}{2^k-1}) \\
&\le& 2s_1,
\end{eqnarray*}
for any $z\in E_k$, we have $E_k \subset \overline{D(a,2s_1)}\setminus\Omega$. By \eqref{eq:transfinite_1}, the $2^k-$th diameter $\delta_{2^k}(\overline{D(a,2s_1)}\setminus\Omega)$ satisfies
\begin{eqnarray*}
\log \delta_{2^k}(\overline{D(a,2s_1)}\setminus\Omega)&=&\log \left(  \sup_{z_1,\cdots,z_{2^k}\in \overline{D(a,2r_1)}\setminus\Omega}\prod_{\mu=1}^{2^k} \prod_{\nu=1,\nu \ne \mu}^{2^k}|z_{\mu}-z_{\nu}|^{\frac{1}{2^k\cdot (2^k-1)}}  \right)\\
&\ge&\log \left( \prod_{l=0}^{k-1}(s_{l+1})^{2^{k-l-1}\cdot\frac{1}{2^k\cdot (2^k-1)} } \right)^{2^k}\\
&=& \frac{1}{ (2^k-1)} \sum_{l=0}^{k-1} 2^{k-l-1}\log s_{l+1}.
\end{eqnarray*}
Thus the Fekete-Szeg\"{o} theorem implies that
\begin{eqnarray} \label{eq2.3}
\notag \log \mathrm{Cap}(\overline{D(a,2s_1)}\setminus\Omega)&=&\lim_{n\to \infty}\log \delta_n(\overline{D(a,2s_1)}\setminus\Omega)\\
&\ge& \sum_{l=0}^{\infty} \frac{\log s_{l+1}}{2^{l+1}}.
\end{eqnarray}

The above argument holds for both $h_i$ ($i=1,2$). When $i=1$, i.e.,
\[
s_{k+1}=\frac{c}{5}s_k^\alpha,
\]
we have
\[
\log s_{k+1}=\alpha^k \log s_1  + \frac{\alpha^k-1}{\alpha-1}\log \frac{c}{5},
\]
so that
\[
\log \mathrm{Cap}(\overline{D(a,2s_1)}\setminus\Omega) \ge \frac{1}{\alpha}\left(\sum_{l=0}^{\infty} \left( \frac{\alpha}{2} \right)^{l+1} \right)\cdot  \log s_1 +O(1),
\]
in view of \eqref{eq2.3}. Thus
\[
\mathrm{Cap}(\overline{D(a,2s_1)}\setminus\Omega) \gtrsim (2s_1)^{\frac{1}{2-\alpha}},
\]
when $0<s_1\ll1$, and the implicit constant does not depends on $a$. This proves $(U)_{1,\alpha}\Rightarrow(C)_{1,(2-\alpha)^{-1}}$.

Suppose that $i=2$. We notice that for $0<t\ll1$,
\[
\log\frac{5}{ct}+\beta\log\log\frac{1}{t}<2\log\frac{1}{t}.
\]
When $0<s_1\ll1$, by using \eqref{eq2.1} repeatedly
\begin{eqnarray*}
s_{k+1} &=& \frac{c}{5}s_k\left(\log \frac{1}{s_k}\right)^{-\beta}\\
&=& \left(\frac{c}{5}\right)^2s_{k-1}\left(\log \frac{1}{s_{k-1}}\right)^{-\beta} \cdot \left( \log \frac{5}{cs_{k-1}}+\beta \log \log \frac{1}{s_{k-1}} \right)^{-\beta}\\
&\geq& \left(\frac{c}{5}\right)^2 \cdot 2^{-\beta} s_{k-1}\left(\log \frac{1}{s_{k-1}}\right)^{-2\beta}\\
&=& \left(\frac{c}{5}\right)^3 \cdot 2^{-\beta} s_{k-2}\left(\log \frac{1}{s_{k-2}}\right)^{-\beta}\cdot \left(  \log \frac{5}{cs_{k-2}}+\beta \log \log \frac{1}{s_{k-2}}       \right)^{-2\beta}\\
&\geq&  \left( \frac{c}{5} \right)^3\cdot 2^{-(\beta+2\beta)} s_{k-2}\left(\log \frac{1}{s_{k-2}}\right)^{-3\beta}\\
&\cdots&\\
&\geq& \left( \frac{c}{5} \right)^k\cdot 2^{-(1+2+\cdots+k-1)\beta}s_1\left(\log \frac{1}{s_1}\right)^{-k\beta}\\
&=& \left( \frac{c}{5} \right)^k\cdot 2^{-\frac{k(k-1)}{2}\beta}s_1\left(\log \frac{1}{s_1}\right)^{-k\beta}.
\end{eqnarray*}
This together with \eqref{eq2.3} imply that
\begin{eqnarray*}
\log \mathrm{Cap}(\overline{D(a,2s_1)}\setminus\Omega)&\ge& \sum_{l=0}^{\infty}\frac{1}{2^{l+1}}\left[ \log s_1-l\log\left[\frac{c}{5}\left(\log \frac{1}{s_1}\right)^{-\beta}\right]-\frac{\beta}{2}l(l-1)\log 2 \right]\\
&=&\left( \sum_{l=0}^{\infty}\frac{1}{2^{l+1}} \right)\cdot \log s_1 + \left( \sum_{l=0}^{\infty}\frac{l}{2^{l+1}} \right)\cdot \log\left[\frac{c}{5}\left(\log \frac{1}{s_1}\right)^{-\beta}\right]+O(1)\\
&=& \log\left[\frac{c}{5}s_1\left(\log \frac{1}{s_1}\right)^{-\beta}\right]+O(1).
\end{eqnarray*}
Thus $(C)_{2,\beta}$ holds.
\end{proof}

Let us consider the Cantor-type set $\mathcal{C}$ defined as follows. Given a sequence $\{l_j\}^\infty_{j=0}$ of positive numbers with $l_{j+1}<l_j/2$, set $\mathcal C_0:=[0,l_0]$ and define $\mathcal C_j$ to be a union of $2^j$ closed intevals inductively, such that $\mathcal C_j$ is obtained by removing from the middle of each inteval in $\mathcal C_{j-1}$ an open subinteval whose length is $l_{j-1}-2l_j$. For example, $\mathcal C_1=[0,l_1]\cup[l_0-l_1,l_0]$, $\mathcal C_2=[0,l_2]\cup[l_1-l_2,l_1]\cup[l_0-l_1,l_0-l_1+l_2]\cup[l_0-l_2,l_0]$, etc. Write
\[
\mathcal C_j=\bigcup^{2^j}_{k=1}I_{j,k},
\]
where every $I_{j,k}$ is a closed inteval of length $l_j$, lying on the left of $I_{j,k+1}$. We set
\[
\mathcal C:=\bigcap^\infty_{j=0} \mathcal C_j.
\]
The (logarithm) capacity of $\mathcal{C}$ satisfies (cf. \cite[Theprem 5.3.7]{Ransford})
\begin{equation}\label{eq:capacity_Cantor}
\mathrm{Cap}(\mathcal{C})\leq\frac{1}{2}\prod^\infty_{j=0}\left(\frac{2l_{j+1}}{l_j}\right)^{1/2^j}.
\end{equation}
Suppose that $l_{j+1}=l_j^\alpha$ and $0<l_0\ll1$. Then $l_j=l_0^{\alpha^j}$ and
\[
\left(\frac{2l_{j+1}}{l_j}\right)^{1/2^j}=2l_0^{\frac{(\alpha/2)^j}{\alpha-1}}.
\]
Thus it follows from \eqref{eq:capacity_Cantor} that $\mathrm{Cap}(\mathcal{C})=0$ when $\alpha\geq2$, i.e., $\mathcal{C}$ is a polar set. On the other hand, we have
\begin{proposition}\label{prop:U1_Cantor}
$\Omega:=\mathbb{C}\setminus\mathcal{C}$ satisfies $(U)_{1,\alpha}$. 
\end{proposition}

\begin{proof}
Clearly, we have $\partial\Omega=\mathcal{C}$. It suffices to find constants $c>0$ and $r_0>0$, such that
\[
\mathcal C\cap\{x\in\mathbb{R};\ cr^\alpha\leq|x-a|\leq r\}\neq\emptyset
\]
for all $a\in{\mathcal C}$ and $0<r<r_0$. We take $0<r_0<2l_0$. It follows that for any $0<r<r_0$, there exists an integer $j\geq0$, such that $l_{j+1}<r/2\leq{l_j}$. Fix $r$ and $j$. By definition, for any $a\in\mathcal{C}$, there is an integer $k$ such that $a\in{I_{j+1,k}}$. We claim that
\begin{equation}\label{eq:nonemptyintersection}
  I_{j+1,k}\cap\{x\in\mathbb{R};\ cr^\alpha\leq|x-a|\leq r\}\neq\emptyset
\end{equation}
for some $0<c<2^{-1-\alpha}$. To see this, suppose on the contrary that the intersection in \eqref{eq:nonemptyintersection} is empty. Since $a\in{I_{j+1,k}}$, we must have $I_{j+1,k}\subset(a-cr^\alpha,a+cr^\alpha)$. However, we infer from the choice of $j$ and $c$ that
\[
2cr^\alpha \leq 2^{1+\alpha}cl_j^\alpha = 2^{1+\alpha}cl_{j+1} < l_{j+1}
\]
i.e., the length of $(a-cr,a+cr)$ is less than the length of $I_{j+1,k}$, which is a contradiction.

In view of \eqref{eq:nonemptyintersection}, we see that if $\mathcal C\cap\{x\in\mathbb{R};\ cr^\alpha\leq|x-a|\leq r\}=\emptyset$, then
\[
(I_{j+1,k}\setminus{\mathcal C})\cap\{x\in\mathbb{R};\ cr^\alpha\leq|x-a|\leq r\}\neq\emptyset.
\]
Thus either $[a-r,a-cr^\alpha]$ or $[a+cr^\alpha,a+r]$ must be contained in a connected component of $I_{j+1,k}\setminus{\mathcal C}$. In particular, one of these two open intevals is contained in $I_{j+1,k}$. However, since $\alpha>1$, we have $c<1/2$. Thus the length of $[a-r,a-cr^\alpha]$ and $[a+cr^\alpha,a+r]$ both equal to $r-cr^\alpha$, which is no less than $r-r^\alpha/2>r/2>l_{j+1}$. This leads to a contradiction.
\end{proof}

In the remaining part of this section, let us verify the proof of Proposition \ref{prop:Zalcman_weak_uniformly_perfect}. Again, we denote $h_1=h_{1,\alpha}$ and $h_2=h_{2,\beta}$ for simplicity.

\begin{proof}
Let $\Omega$ be the Zalcman-type domain defined by $h_1$. We first show that $\Omega$ satisfies the condition $(U)_{1,\alpha}$. Recall that $h_1(r)=r^\alpha$ and $x_{k+1}=r_{k+1}=h_1(x_k)$. For simplicity, we may assume that $x_1$ is sufficiently small, so that
\begin{equation}\label{eq:h_1_weakly_uniform_perfect}
h_1(r)<\frac{r}{6}<\frac{r}{2},\ \ \ \forall\,0<r<x_1,
\end{equation}
i.e., $x_{k+1}=r_k<x_k/6<x_k/2$. Take $a\in\partial\Omega$. We shall divide the argument into three cases:

(i) $a=0$. Let $0<r<x_1$. Then there exists an integer $k$, such that $x_{k+1}+r_{k+1}<r\le x_k+r_k$. If $x_{k}-r_k<r<x_{k}+r_k$, then clearly
\[
\{ z;\ h_1(r)\leq|z|\leq r  \} \cap \partial \Omega \ne \emptyset.
\]
On the other hand, if $x_{k+1}+r_{k+1}<r\leq x_k-r_k$, since $h_1$ is increasing, we see that
\[
h_1(r) \le h_1(x_k)=x_{k+1} \le x_{k+1}+r_{k+1},
\]
so that
\[
x_{k+1}+r_{k+1}\in\{z;\ h_1(r)\le |z|\le r  \} \cap \partial \Omega.
\]

(ii) $a\in\partial{D}(0,1)$. Then we have
\begin{eqnarray*}
& & \{ z;\ h_1(r)\leq |z-a|\leq r  \} \cap \partial \Omega\\
&\supset& \{ z;\ h_1(r)\leq |z-a|\leq r  \} \cap \partial D(0,1)\\
&\neq& \emptyset
\end{eqnarray*}
when $0<r<1$.

(iii) $a\in\partial{D}(x_k,r_k)$ for some $k$. We fix a positive integer $k_0$, and let $0<r<x_{k_0}/2$. It follows that $h_1(r)<h_1(x_{k_0})=r_{k_0}$, and hence
\[
\{ z;\ h_1(r)\leq |z-a|\leq r  \} \cap \partial D(x_k,r_k)\neq\emptyset,\ \ \ 1\leq{k}\leq{k_0}.
\]
Next, suppose that $k>k_0$. For any $r\in(0,x_{k_0}/2)$, we can always find $x_m$, with $k_0\leq{m}<k$, such that $x_{m+1}-x_k<r\leq x_m-x_k$. It follows that
\[
\frac{1}{2}h_1(r)\leq\frac{1}{2}h_1(x_m)=\frac{1}{2}x_{m+1}.
\]
If $m<k-1$, i.e., $m+1<k$, we have
\[
x_{m+1}-r_{m+1}-x_k>x_{m+1}-\frac{1}{2}x_{m+1}>0.
\]
Thus
\begin{eqnarray*}
|x_{m+1}-r_{m+1}-a| &=& |x_{m+1}-r_{m+1}-x_k-(a-x_k)|\\
&\geq&  x_{m+1}-r_{m+1}-x_k-|a-x_k|\\
&=& x_{m+1}-r_{m+1}-x_k-r_k\\
&>& x_{m+1}-\frac{3}{6}x_{m+1}\\
&=& \frac{1}{2}x_{m+1}\\
&>& \frac{1}{2}h_1(r),
\end{eqnarray*}
in view of \eqref{eq:h_1_weakly_uniform_perfect}, and
\begin{eqnarray*}
|x_{m+1}-r_{m+1}-a| &=& |x_{m+1}-r_{m+1}-x_k-(a-x_k)|\\
&\leq&  x_{m+1}-r_{m+1}-x_k+|a-x_k|\\
&=& x_{m+1}-r_{m+1}-x_k+r_k\\
&<& x_{m+1}-x_k\\
&<& r.
\end{eqnarray*}
That is
\[
x_{m+1}-r_{m+1}\in\left\{z;\ \frac{1}{2}h_1(r)\leq|z-a|\leq{r}\right\}\cap\partial\Omega.
\]
If $m=k-1$, we have $0<r\leq{x_{k-1}-x_k}$. If $x_k+r_k<r\leq{x_{k-1}-x_k}$, we have
\[
\frac{h_1(r)}{2}<\frac{x_k}{2}<x_k-r_k<|0-a|<x_k+r_k<r,
\]
i.e.,
\[
0\in\left\{z;\ \frac{1}{2}h_1(r)\leq|z-a|\leq{r}\right\}\cap\partial\Omega,
\]
while if $r\leq{x_k+r_k}$, we have
\[
h_1(r)\leq{h_1(x_k+r_k)}=(x_k+r_k)^\alpha<2^\alpha{x_k^\alpha}=2^\alpha{r_k},
\]
and hence
\begin{eqnarray*}
& & \left\{ z;\ \frac{1}{2^\alpha}h_1(r) \leq |z-a|\leq r  \right\} \cap \partial \Omega\\
&\supset& \left\{ z;\ \frac{1}{2^\alpha}h_1(r)\leq |z-a|\leq r  \right\} \cap \partial D(x_k,r_k)\\
&\neq& \emptyset.
\end{eqnarray*}
Thus we have proved that $\Omega$ satisfies the condition $(U)_{1,\alpha}$.

Next, let us verify that $\Omega$ does not satisfy the condition $(U)_{1,\alpha-\varepsilon}$. To see this, we notice that
\[
\frac{1}{k}\left(\frac{x_k}{2}\right)^{\alpha-\varepsilon}>2x_k^\alpha=2x_{k+1}
\]
when $k\gg1$. Thus
\begin{eqnarray*}
& & \left\{ z;\ \frac{1}{k}\left(\frac{x_k}{2}\right)^{\alpha-\varepsilon} \leq |z| \leq \frac{x_k}{2}  \right\} \cap \partial \Omega\\
&\subset& \left\{ z;\ 2x_{k+1} \leq |z| \leq \frac{x_k}{2}  \right\} \cap \partial \Omega\\
&=& \emptyset.
\end{eqnarray*}
This completes the proof of (1).

The proof of (2) is completely analogous, and we leave it for the readers.
\end{proof}

\section{Proof of Theorem \ref{th:Bergman_kernel}}\label{sec:main}

Let $E\subset\mathbb{C}$ be a non-polar compact subset and $\mu_E$ its equilibrium measure. Following \cite{Zwonek,PflugZwonek}, we set
\begin{equation}\label{eq:f_E_def}
f_E(w):=\int_E\frac{d\mu_E(\zeta)}{w-\zeta}.
\end{equation}
It follows that $f_E$ is a holomorphic function on $\mathbb{C}\setminus{E}$. Moreover, Lemma 2 in \cite{PflugZwonek} indicates that

\begin{lemma}\label{lm:f_E_integral}
If $E$ is a non-polar compact subset in $D(0,1/4)$, then 
\[
\int_{D(0,1/4)\setminus{E}}|f_E|^2\lesssim \log \frac{1}{\mathrm{Cap}(E)}.
\]
\end{lemma}

By Lemma \ref{lm:f_E_integral}, if $E\subset{D(0,r)}$ with $0<r<1/4$, we have
\begin{equation}\label{eq:f_E_0}
\int_{D(0,r)\setminus{E}}|f_E|^2\lesssim \log \frac{1}{\mathrm{Cap}(E)}.
\end{equation}
Moreover, by dilatation, we infer from \eqref{eq:measure_dilatation} that
\begin{eqnarray*}
f_{tE}(w)= \int_{tE}\frac{d\mu_{tE}(\zeta)}{w-\zeta} = \int_{tE}\frac{d\mu_{E}(t^{-1}\zeta)}{w-\zeta} = \int_E\frac{d\mu_E(\zeta)}{w-t\zeta} = \frac{f_E(w/t)}{t},\ \ \ \forall\,t>0,
\end{eqnarray*}
and hence
\begin{eqnarray}\label{eq:f_E}
\int_{D(0,r)\setminus{E}}|f_E|^2
 &=& 16r^2\int_{D(0,1/4)\setminus{\frac{1}{4r}E}}|f_E(4rw)|^2\nonumber\\
 &=& \int_{D(0,1/4)\setminus{\frac{1}{4r}E}}|f_{\frac{1}{4r}E}|^2\nonumber\\
 &\lesssim& \log \frac{4r}{\mathrm{Cap}(E)}.
\end{eqnarray}
in view of \eqref{eq:cap_dilatation}. The inequality \eqref{eq:f_E_0} suffices to prove Theorem \ref{th:Bergman_kernel}/(1), while we need the sharper inequality \eqref{eq:f_E} for the second assertion.

\begin{proof}[Proof of Theorem \ref{th:Bergman_kernel}/(1)]
Let $h_1(t)=h_{1,\alpha}(t)=t^\alpha$. Let $w\in\Omega$ be sufficiently close to $\partial\Omega$, and $w'\in\partial\Omega$ such that $|w-w'|=\delta_\Omega(w)$. Take $r>0$ so that $c\cdot{h}_1(r)=8\delta_\Omega(w)$. It follows that from the condition $(U)_{1,\alpha}$ that there exists another point $w''\in\partial\Omega$ with
\[
8\delta_\Omega(w)\leq|w''-w'|\leq{r}.
\]
Let $E_1:=\overline{D(w',\delta_\Omega(w))}\setminus\Omega$. As in \cite{PflugZwonek}, we divide $E_1$ into the following three parts:
\begin{eqnarray*}
E_{11} &=& E_1\cap\left\{w+se^{i\theta}\in\mathbb{C};\ s>0,\ -\frac{\pi}{3}\leq\theta\leq\frac{\pi}{3}\right\},\\
E_{12} &=& E_1\cap\left\{w+se^{i\theta}\in\mathbb{C};\ s>0,\ \frac{\pi}{3}\leq\theta\leq\pi\right\},\\
E_{13} &=& E_1\cap\left\{w+se^{i\theta}\in\mathbb{C};\ s>0,\ \pi\leq\theta\leq\frac{5\pi}{3}\right\},
\end{eqnarray*}
so that
\begin{equation}\label{eq:cos}
\cos(\mathrm{arg}(\zeta-w))\geq\frac{1}{2},\ \ \ \zeta\in{E_{11}}.
\end{equation}
Moreover, by rotating $\Omega$ around $w$, we may assume that $\mathrm{Cap}(E_{11})\geq\mathrm{Cap}(E_{12})\geq\mathrm{Cap}(E_{13})$. Then we infer from \eqref{eq:cap_sum} that
\begin{eqnarray*}
\frac{1}{\log ({1}/{\mathrm{Cap}(E_1)})} &\le& \frac{1}{\log ({1}/{\mathrm{Cap}(E_{11})})}+\frac{1}{\log ({1}/{\mathrm{Cap}(E_{12})})}+\frac{1}{\log ({1}/{\mathrm{Cap}(E_{13})})}\\
&\le& \frac{3}{\log ({1}/{\mathrm{Cap}(E_{11})})}.
\end{eqnarray*}
so that
\begin{equation}\label{eq:log_E_11}
\log \frac{1}{\mathrm{Cap}(E_{11})} \leq 3\log \frac{1}{\mathrm{Cap}(E_1)}
\end{equation}

Set $E_2:=\overline{D(w'',\delta_\Omega(w))}\setminus\Omega$. We have $E_{11},E_2\subset{D(w,2r)}$. Let
\[
f_{11}:=f_{E_1},\ \ \ f_2:=f_{E_2},
\]
be the functions given in \eqref{eq:f_E_def}, and consider $f:=f_{11}-f_2$. Write
\[
\int_{\Omega}|f|^2 = \int_{\Omega\cap D(w,2r)}|f|^2+\int_{\Omega\setminus{D(w,2r)}}|f|^2=:I_1+I_2.
\]
For $I_1$, it follows from \eqref{eq:f_E_0} and \eqref{eq:log_E_11} that
\begin{eqnarray}
I_1 &\lesssim& \int_{\Omega\cap D(w,2r)}\left(|f_{11}|^2+|f_2|^2\right)\nonumber\\
&\le& \int_{D(w,2r)\setminus{E_{11}}}|f_{11}|^2+\int_{D(w,2r)\setminus{E_2}}|f_2|^2\nonumber\\
&\lesssim& \log \frac{1}{\mathrm{Cap}(E_{11})} + \log \frac{1}{\mathrm{Cap}(E_2)}\nonumber\\
&\lesssim& \log \frac{1}{\mathrm{Cap}(E_1)}+\log \frac{1}{\mathrm{Cap}(E_2)}\label{eq:I_1}
\end{eqnarray}
By Theorem \ref{th:weakly_uniform_perfect}/(1), we have
\[
\mathrm{Cap}(E_1)\gtrsim\delta_\Omega(w)^{(2-\alpha)^{-1}},\ \ \ \mathrm{Cap}(E_2)\gtrsim\delta_\Omega(w)^{(2-\alpha)^{-1}}.
\]
so that
\begin{equation}\label{eq:I_1-1}
I_1 \lesssim \log\frac{1}{\delta_\Omega(w)}.
\end{equation}
For $I_2$, straightforward computation gives
\begin{eqnarray*}
I_2 &\leq& \int_0^{2\pi} \int_{2r}^{\infty} \left|\int_{E_{11}}\frac{d\mu_{E_{11}}(\zeta)}{w+se^{i\theta}-\zeta}-\int_{E_{2}}\frac{d\mu_{E_{2}}(\zeta')}{w+se^{i\theta}-\zeta'} \right|^2 sdsd\theta \\
&=& \int_0^{2\pi} \int_{2r}^{\infty}\left| \int_{E_{11}}\int_{E_2}\frac{1}{w+se^{i\theta}-\zeta}-\frac{1}{w+se^{i\theta}-\zeta'} d\mu_{E_{11}}(\zeta)d\mu_{E_2}(\zeta')  \right|^2 sdsd\theta\\
&=& \int_0^{2\pi} \int_{2r}^{\infty}\left| \int_{E_{11}}\int_{E_2}\frac{\zeta-\zeta'}{(w+se^{i\theta}-\zeta)(w+se^{i\theta}-\zeta')} d\mu_{E_{11}}(\zeta)d\mu_{E_2}(\zeta')  \right|^2 sdsd\theta\\
&\le& 2\pi \int_{2r}^{\infty}  \frac{(2r)^2}{|s-r-2\delta_\Omega(w)|^4} sds \\
&=& 8\pi r^2 \int_{r-2\delta_\Omega(w)}^{\infty} \frac{s+r+2\delta_\Omega(w)}{s^4}\mathrm{d}s\\
&\le& 24\pi{r^2} \int_{r/2}^{\infty} \frac{1}{s^3}\mathrm{d}s\\
&\le& 48\pi\\
&<& \infty.
\end{eqnarray*}
This combined with \eqref{eq:I_1-1} yield
\begin{equation}\label{eq:f_L2_thm2}
\int_\Omega|f|^2\lesssim\log\frac{1}{\delta_\Omega(w)}.
\end{equation}

It remains to find a lower bound for $|f(w)|$. Clearly, we have $|f(w)|\geq|f_{11}(w)|-|f_2(w)|$. First, by using \eqref{eq:cos} and noticing that $E_{11}\subset{D(w,2\delta_\Omega(w))}$, we have
\begin{eqnarray*}
|f_{11}(w)|&=&\left| \int_{E_{11}}\frac{d\mu_{K_{11}}(\zeta)}{w-\zeta}  \right|\\
&=& \left| \int_{E_{11}} \frac{\cos(\arg(\zeta-w))}{|w-\zeta|} d\mu_{K_{11}}(\zeta) +i \int_{E_{11}} \frac{\sin(\arg(\zeta-w))}{|w-\zeta|} d\mu_{K_{11}}(\zeta) \right| \\
&\ge&\left| \int_{E_{11}} \frac{\cos(\arg(\zeta-w))}{|w-\zeta|} d\mu_{E_{11}}(\zeta)  \right| \\
&\ge& \frac{1}{2}\int_{E_{11}}\frac{1}{|w-\zeta|}d\mu_{E_{11}}(\zeta)\\
&\ge& \frac{1}{4\delta_\Omega(w)}.
\end{eqnarray*}
On the other hand,
\begin{eqnarray*}
|f_2(w)| &\le& \int_{E_2} \frac{d\mu_{K_2}(\zeta)}{|w-\zeta|}\\
&\le& \int_{E_2} \frac{d\mu_{K_2}(\zeta)}{|w'-w''|-|w-w'|-|w''-\zeta|}\\
&\le& \int_{E_2} \frac{d\mu_{K_2}(\zeta)}{8\delta_\Omega(w)-\delta_\Omega(w)-\delta_\Omega(w)}\\
&\le& \frac{1}{6\delta_\Omega(w)}.
\end{eqnarray*}
Thus $|f(w)|\geq1/(12\delta_\Omega(w))$, which together with \eqref{eq:f_L2_thm2} complete the proof of Theorem \ref{th:Bergman_kernel}/(1).
\end{proof}

\begin{remark}
{\rm If $\Omega\subset\mathbb{C}$ is a bounded domain, Pflug-Zwonek proved that
\[
K_\Omega(w)\gtrsim\gamma_\Omega(w),
\]
where $\gamma_\Omega$ is a potential theoretic function defined by
\[
\gamma_\Omega(w):=\int^{1/4}_0\frac{dr}{r^3\log(1/\mathrm{Cap}(\overline{D(w,r)}\setminus\Omega))}.
\]
Let $a\in\partial\Omega$ with $\delta_\Omega(z)=|w-a|$. We have
\[
\overline{D(w,r)}\setminus\Omega\supset\overline{D(z,r-\delta_\Omega(w))}\setminus\Omega,
\]
so that
\[
\mathrm{Cap}(\overline{D(w,r)}\setminus\Omega)\gtrsim(r-\delta_\Omega(z))^{(2-\alpha)^{-1}},
\]
in view of Theorem \ref{th:weakly_uniform_perfect}/(2) when $\delta_\Omega(w)<r$. Thus
\[
\gamma_\Omega(w) \geq \int^{1/4}_{2\delta_\Omega(w)}\frac{dr}{r^3\log\frac{1}{r-\delta_\Omega(w)}}\geq \frac{1}{\log\frac{1}{\delta_\Omega(w)}}\int^{1/4}_{2\delta_\Omega(w)}\frac{dr}{r^3}\gtrsim \frac{1}{\delta_\Omega(w)^2\log\frac{1}{\delta_\Omega(w)}}.
\]
However, this argument does not yield a sharper lower bound for $K_\Omega$ even if $\Omega$ is uniformly perfect.}
\end{remark}

Next, we modify the proof of Theorem \ref{th:Bergman_kernel}/(1) to prove the second assertion. Replace the function $h_1(t)=t^\alpha$ by $h_2(t)=t(\log(1/t))^{-\beta}$, take $w,w'$ as above, and suppose that $c\cdot{h}_2(r)=8\delta_\Omega(w)$. We proceed as in the previous proof, and the only difference will occur in the upper estimate \eqref{eq:I_1-1} for $I_1$. If $\Omega$ satisfies the condition $(U)_{2,\beta}$, then Theorem \ref{th:weakly_uniform_perfect}/(2) implies that
\begin{equation}\label{eq:cap_thm2}
\mathrm{Cap}(E_1)\gtrsim\delta_\Omega(w)\left(\log\frac{1}{\delta_\Omega(w)}\right)^{-\beta},\ \ \ \mathrm{Cap}(E_2)\gtrsim\delta_\Omega(w)\left(\log\frac{1}{\delta_\Omega(w)}\right)^{-\beta}.
\end{equation}
Moreover, if $g_2$ is the inverse function of $h_2$, then we have
\begin{equation}\label{eq:r_g_2}
r=g_2(8\delta_\Omega(w)/c).
\end{equation}
We need the following lemma.

\begin{lemma}\label{lm:g_2}
$g_2(t) \leq t(\log(1/t))^\beta$ when $t$ is sufficiently small.
\end{lemma}
\begin{proof}
We have
\begin{eqnarray*}
h_2(t)\left(\log\frac{1}{h_2(t)}\right)^{\beta} &=& t\left(\log\frac{1}{t}\right)^{-\beta}\left(\log\frac{1}{t(\log(1/t))^{-\beta}}\right)^{\beta}\\
&=& t\left(\log\frac{1}{t}\right)^{-\beta}\left(\log \frac{1}{t} +\beta\log\log\frac{1}{t}\right)^\beta\\
&=& t\left(1 +\beta\frac{\log\log(1/t)}{\log(1/t)}\right)^\beta\\
&\ge& t\\
&=& g_2(h_2(t)).
\end{eqnarray*}
Since $h_2(t)$ is increasing when $t$ is sufficiently small, so is $g_2$, and hence $g_2(t) \leq t(\log(1/t))^\beta$ when $0<t\ll1$.
\end{proof}

Therefore, \eqref{eq:r_g_2} implies that
\begin{equation}\label{eq:r_case2}
r\lesssim\delta_\Omega(w)\left(\log\frac{1}{\delta_\Omega(w)}\right)^\beta.
\end{equation}
By using \eqref{eq:f_E} instead of \eqref{eq:f_E_0}, the estimate \eqref{eq:I_1} can be modified to be
\begin{eqnarray*}
I_1 &\lesssim& \int_{D(w,2r)\setminus{E_{11}}}|f_{11}|^2+\int_{D(w,2r)\setminus{E_2}}|f_2|^2\\
&\lesssim& \log \frac{8r}{\mathrm{Cap}(E_{11})} + \log \frac{8r}{\mathrm{Cap}(E_2)}
\end{eqnarray*}
Similarly to \eqref{eq:log_E_11}, we  have
\[
\log \frac{8r}{\mathrm{Cap}(E_{11})} \leq 3\log \frac{8r}{\mathrm{Cap}(E_1)},
\]
so that
\[
I_1 \lesssim \log \frac{8r}{\mathrm{Cap}(E_1)}+\log \frac{8r}{\mathrm{Cap}(E_2)}.
\]
This together with \eqref{eq:cap_thm2} and \eqref{eq:r_case2} yield the following estimate for $I_1$:
\begin{equation}\label{eq:I_1-2}
I_1 \lesssim \log\log\frac{1}{\delta_\Omega(w)}.
\end{equation}
The other parts of the proof of Theorem \ref{th:Bergman_kernel}/(1) can be repeated without change. We leave it to the reader.

{\bf Acknowledgements.} We are grateful to Prof. Bo-Yong Chen for introducing this topic to us and many inspiring discussions.

\end{document}